\documentclass[12pt]{amsart}
\usepackage{amssymb,amstext,amsmath,amscd,amsthm,amsfonts,enumerate,graphicx,latexsym}
\newtheorem{theorem}{Theorem}[section]
\newtheorem{lemma}[theorem]{Lemma}
\newtheorem{corollary}[theorem]{Corollary}
\newtheorem{proposition}[theorem]{Proposition}
\theoremstyle{definition}
\newtheorem{definition}[theorem]{Definition}

\newtheorem{example}[theorem]{Example}
\newtheorem{remark}[theorem]{Remark}
\def\AA{\mathcal A} 
\def\CB{\mathrm CB} 
\def\CC{\mathcal C} 
\def\cl{\gamma \circ \Sigma}
\def\CMR{{\mathcal C} (R)}

\def\End{\mathrm{End}}
\def\sEnd{\underline{\mathrm{End}}}
\def\FF{\mathcal{F}}
\def\GG{\mathcal{G}}
\def\Hom{\mathrm{Hom}}
\def\Im{\mathrm{Im}}

\def\KGdim{\mathrm{KGdim}\, }
\def\m{\mathfrak m}

\def\mod{\mathrm{mod}}

\def\modCC{{\mathrm{mod}}\, \CC }
\def\p{\mathfrak p}

\def\rad{{\mathrm{rad}}}
\def\SS{{\mathcal S}}

\def\sCC{\underline{{\mathcal C}}}
 
\def\sCMR{\underline{{\mathcal C}}(R)} 

\def\sEnd{\underline{\mathrm{End}}}
\def\sHom{\underline{\mathrm{Hom}}}
\def\sHomR{\underline{\mathrm{Hom}}_R}

\def\modsC{{\mathrm{mod}}\,{\sCMR }}
\def\modsCC{{\mathrm{mod}}\,{\sCC }}
\def\modR{{\mathrm{mod}}\, R}
\def\SpC{{\mathsf{Sp}}\, {\mathcal C}}
\def\SpCM{{\mathsf{Sp}}\, {\CMR}}
\def\Spec{{\mathrm{Spec}}}
\def\srad{\underline{\mathrm{rad}}}

\def\TT{\mathcal{T}}
\def\XX{\mathcal{X}}
\def\YY{\mathcal{Y}}

\begin{document}
\title[The spectrum of a category of MCM modules] {The spectrum of a category of maximal Cohen-Macaulay modules}
\author{Naoya Hiramatsu} 
\address{Institute for the Advancement of Higher Education, Okayama University of Science,
1-1 Ridai-cho, Kita-ku, Okayama 700-0005, Japan}
\email{n-hiramatsu@ous.ac.jp} 
\maketitle
\begin{abstract}
We introduce an analog of the Ziegler spectrum for maximal Cohen-Macaulay modules over a complete Cohen-Macaulay local ring. 
We define a topology on the space of isomorphism classes of indecomposable maximal Cohen-Macaulay modules and investigate the topological structure. 
We also calculate the Cantor-Bendixson rank for a ring which is of $\mathrm{CM}_+$-finite representation type.
\end{abstract}

\section{Introduction}
The Ziegler spectrum of modules was introduced as a model-theoretic perspective on module theory \cite{Z84}. 
It is a topological space whose points are the isomorphism classes of indecomposable pure-injective modules and defined in terms of solution sets to certain types of linear conditions. 
The topology has an equivalent definition, that is, it can be defined in terms of morphisms between finitely presented functors.     
Many studies of the Ziegler spectrum are given in the context of the representation theory of algebras \cite{Her97, Kr97, P18, LP19, Nak22}.  
In this paper, we consider an analog of the Ziegler spectrum for a (stable) category of maximal Cohen-Macaulay (abbr. MCM) modules over a complete Cohen-Macaulay local ring.

Let $R$ be a complete Cohen--Macaulay local ring. 
We denote by $\CMR$ the category of MCM $R$-modules and by $\sCMR$ the stable category of $\CMR$. 
We denote by $\modsC$ the category of finitely presented contravariant additive functors. 
We put $\SpCM$ the set of isomorphism classes of the indecomposable MCM $R$-modules except $R$ and $0$.
For a subset $\XX$ of $\SpCM$, we denote by $\Sigma (\XX )$ the subcategory of $\modsC$ formed by the functors $F$ such that $F(X)=0$ for all $X \in \XX$. 
For a subcategory $\FF$ of $\modsC$, we denote by $\gamma (\FF )$ the subset of $\SpCM$ satisfying $F(X) = 0$ for all $F \in \FF$. 
By using the assignment $\Sigma$ and $\gamma$, we can define a closure operator on $\SpCM$, which is an analog of the Ziegler spectrum.

\begin{theorem}
The assignment 
$
\XX \mapsto \gamma \circ \Sigma (\XX )
$ 
is a is a Kuratowski closure operator on $\SpCM$. 
In particular, it induces a topology on $\SpCM$. 
\end{theorem}

The theorem is proved for a slightly more general category than $\CMR$.  

The Cantor-Bendixson rank $\CB$ measures the complexity of the topology. 
It measures how far the topology is from the discrete topology. 
We say that a Cohen--Macaulay local ring is $\mathrm{CM}_+$-finite if there exist only finitely many isomorphism classes of indecomposable MCM modules that are not locally free on the punctured spectrum. 

\begin{theorem}
If $R$ is $\mathrm{CM}_+$-finite then $\CB (\SpCM ) \leq 1$. 
\end{theorem}

This study was motivated by the work of Herzog\cite{Her97} and Krause\cite{Kr97}. 
Their studies discuss locally coherent subcategories which include infinitely generated modules. 
Our study considers only finitely generated modules, so it differs from their considerations (Remark \ref{A7}).

The organization of this paper is as follows. 
In Section \ref{A} we introduce a topology on $\SpC$ with an additive subcategory of finitely generated $R$-modules $\CC$ that is closed under kernels of epimorphisms and contains the free $R$-modules (Theorem \ref{A5}) and study the topological structure. 
We will give a bijection between closed subsets in $\SpCM$ and suitable Serre subcategories of $\modsC$ if $R$ is an isolated singularity (Proposition \ref{A10}). 
Section \ref{B} is devoted to the computation of the Cantor- Bendixson rank of $\SpCM$ (Theorem \ref{B8} ). 
\section{The spectrum of the category of maximal Cohen-Macaulay modules}\label{A}

In this paper, we always assume that $R$ is a commutative complete Noetherian local ring with algebraic residue field $k$, and all modules are ``finitely generated" $R$-modules. 
We denote by $\modR$ the category of (finitely generated) $R$-modules. 
When we consider a subcategory $\CC$ of $\modR$, we mean that $\CC$ is an additive full subcategory and closed under isomorphisms.
Whenever we write $\CMR$, we always mean a subcategory of $\modR$ consisting of MCM $R$-modules over a Cohen--Macaulay local ring $R$. That is, 
$$
\CMR = \{ M \ | \ \mathrm{Ext} _{R}^{i}(k, M) = 0 \ \text{for all } i < \dim R \} .
$$
Since $R$ is complete, $\modR$, hence $\CC$ and $\CMR$ are Krull-Schmidt categories (cf. \cite[(1.18)]{Y}). 

For an additive category $\CC$, we denote by $\modCC$ the category whose objects are finitely presented contravariant additive functors from $\CC$ to a category of abelian groups and whose morphisms are natural transformations between functors.
$$
\modCC  = \{ F : \CC \to Ab \ | \ \Hom _\CC (\  , N) \to \Hom _\CC (\  , M) \to F \to 0 \text{ with $M, N \in \CC$} \}.
$$ 

Let $\CC$ be a subcategory of $\modR$. 
We denote by $\sCC$ the stable category of $\CC$. 
The objects of $\sCC$ are the same as those of $\CC$, and the morphisms $\sHom _R(M, N):=\Hom _R(M, N)/ \{ M \to P \to N \text{ where  $P$ is free} \}$. 

First, we state a well-known result on $\modsCC$. 
We say that $\CC$ is closed under kernels of epimorphisms if the kernel of each epimorphism $Y \to X \to 0$ with $X, Y \in \CC$ belongs to $\CC$. 

\begin{proposition}\cite[Proposition 3.3]{Y05}
Let $\CC$ be a subcategory of $\modR$ which contains the free $R$-modules and which is closed under kernels of epimorphisms.
Then $\modsCC$ is an abelian category. 
\end{proposition}
  
In our result, the functor category $\modsCC$ needs to be an abelian category. 
Thus, in what follows, we assume that $\CC$ is a subcategory of $\modR$ which is closed under kernels of epimorphisms and contains the free modules.
Bear in mind that category $\CMR$ is closed under kernels of epimorphisms and contains the free modules. 
So $\modsC$ is an abelian category.

\begin{remark}\label{A1}
\begin{itemize}
\item[(1)] We should remark that we have the equivalence of categories
$$
\modsCC \cong  \{ F \in \modCC  | F(R)=0\};\quad F \mapsto F \circ \iota , 
$$
where $\iota : \CC \to \sCC$. 
See \cite[Remark 4.16]{Y}.
By the equivalence, for all $F \in \modsCC$, we have $0 \to L \to M \to N \to 0$ such that 
$
0 \to \Hom _R (- , L) \to \Hom _R (- , M) \to \Hom _R (-  , N) \to F \to 0 
$
is exact in $\modCC$. 

\item[(2)] Let $M \in \CC$ and $F \in \modCC$. 
Then $F(M)$ has a right $\End_R (M)$-module structure (see \cite[Remark 4.2] {Y}). 
Moreover one can also show that, for a morphism $\varphi : \sHom _R (-,X) \to F $ in $\modsCC$, the morphism $\varphi (M) : \sHom _R (M,X) \to F(M)$ is a right $\End_R (M)$-module homomorphism. 
\end{itemize}
\end{remark}

\begin{definition}\label{A2}
We denote by $\SpC$ the set of isomorphism classes of the indecomposable $R$-modules in $\CC$ except $R$ and $0$.
$$
\SpC :=\{ \text{an indecomposable $R$-modules in $\CC$ except $R$ and $0$} \}/ \cong 
$$
\end{definition}

The following assignments are due to Krause\cite{Her97}, which play key rolls in this paper.

\begin{definition}\label{A3}\cite{Her97, Kr97}
The assignments
$$
\Sigma : \SpC  \to \modsCC, \quad \gamma : \modsCC \to \SpC
$$
are defined by
$$
\begin{array}{l}
\Sigma (\XX ) := \{ F \in \modsCC \ | \ F(X)=0 \ \text{for all $X \in \XX$} \}\\
\gamma (\FF ) := \{ M \in \SpC \ | \ F(M)=0 \ \text{for all $F \in \FF$} \}.
\end{array}
$$
\end{definition}

We state several basic properties of the assignments $\Sigma$ and $\Gamma$.

\begin{lemma}\label{A4}
Let $\XX$, $\YY$ be subsets of $\SpC$ and $\FF$ and $\GG$ be subcategories of $\modsCC$. 
For the assignments $\Sigma$ and $\gamma$, the following statements hold. 
\begin{itemize}
\item[(1)] If $\XX \subseteq \YY$ then $\Sigma (\XX ) \supseteq \Sigma (\YY )$.

\item[(2)] If $\FF \subseteq \GG$ then $\gamma (\FF ) \supseteq \gamma (\GG )$.

\item[(3)] A subset $\XX$ is contained in $\gamma \circ \Sigma (\XX )$. 
Moreover $\Sigma (\XX ) = \Sigma \circ \gamma \circ \Sigma (\XX )$. 

\item[(4)] A subcategory $\FF$ is contained in $\Sigma \circ \gamma (\FF )$. 
Moreover $\gamma (\FF ) = \gamma \circ \Sigma \circ \gamma (\FF )$. 

\item[(5)] $\Sigma (\XX )$ is a Serre subcategory in $\modsCC$. 
\end{itemize}
\end{lemma}

\begin{proof}
We show (1), (3) and (5). 
(2) and (4) follow similarly to (1) and (3) respectively. 
Suppose that $\XX \subseteq \YY$. 
For a functor $F \in \Sigma (\YY )$, $F(\YY ) =0$. 
Thus $F(\XX ) =0$, which $F$ belongs to $\Sigma (\XX )$. 
This shows (1). 
By the definition of $ \Sigma (\XX )$, $\XX$ is contained in $\gamma (F )$ for all functors $F \in \Sigma (\XX )$. 
Thus implies that $\XX \subseteq \gamma \circ \Sigma (\XX )$. 
The inclusion $\Sigma (\XX ) \supseteq \Sigma \circ \gamma \circ \Sigma (\XX )$ holds by (2). 
Take a functor $F \in \Sigma (\XX )$. 
By the definition of $\gamma \circ \Sigma (\XX )$, $F(\gamma \circ \Sigma (\XX )) =0$. 
Hence $F$ belongs to $\Sigma \circ \gamma \circ \Sigma (\XX )$, so that (3) holds. 
Consider an exact sequence in $\modsC$: $0 \to H \to G \to F \to 0$. 
Then one obtains the exact sequence $0 \to H(X) \to G(X) \to F(X) \to 0$ for each $X \in \XX$, so it is clear that $\Sigma (\XX )$ is a Serre subcategory. 
\end{proof}

We state the main theorem of this paper.

\begin{theorem}\label{A5}
The assignment $\XX \mapsto \cl (\XX)$ is a Kuratowski closure operator. 
That is, 
\begin{itemize}
\item[(1)] $\cl (\emptyset) = \emptyset $,
\item[(2)] $\mathcal{X} \subseteq \cl(\mathcal{X} )$, 
\item[(3)] $\cl (\mathcal{X} \cup \mathcal{Y} ) = \cl (\mathcal{X} ) \cup \cl (\mathcal{Y})$, 
\item[(4)] $\cl (\cl ( \mathcal{X}) ) = \cl (\mathcal{X} )$
\end{itemize}
hold for all subsets $\XX$, $\YY$ in $\SpC$. 
\end{theorem}

\begin{proof}
The assertions (1), (2), and (4) follow from the definition and Lemma \ref{A4}. 
To show (3), we now notice that $\sHomR (-, M) \in \modsCC$ for all $M \in \CC$. 
The inclusion $\cl (\XX \cup \YY ) \supseteq \cl (\XX ) \cup \cl (\YY)$ follows from the fact that $\Sigma (\XX \cup \YY ) = \Sigma(\XX ) \cap \Sigma (\YY )$, and the equality is clear. 
To show another inclusion, we take $M \in \cl  (\XX \cup \YY )$. 
Note that $M$ is indecomposable. 
Assume that $M \not\in \cl (\XX ) \cup \cl (\YY)$. 
Then there exist $F \in \Sigma (\XX )$ and $G \in \Sigma (\YY )$ such that  $F(M) \not=0$ and $G(M) \not=0$. 
By Yoneda's Lemma, we have nonzero morphisms $f : \sHom _R (-, M) \to F$ and $g : \sHom _R (-, M) \to G$.
Take a pushout diagram in $\modsCC$:
$$\begin{CD}
\sHom _R (-, M)@>>>  \Im \, f @>>>0\\
@VVV @VVV  \\
 \Im \, g @>>> H @>>> 0\\
 @VVV @VVV  \\
 0 @. 0. 
\end{CD}
$$
Since $\Sigma (\XX)$ and $\Sigma (\YY)$ are Serre subcategories, $\Im \, f \in \Sigma (\XX)$, $\Im \, g \in \Sigma (\YY)$. 
This implies that $H \in \Sigma (\XX \cup \YY)$. 
From the push out diagram we obtain the exact sequence $\sHom _R (-, M) \to \Im \, f \oplus \Im \, g \to H \to 0$.  
Since $\sEnd_R(M)$ is local, $\sEnd_R(M)$ is an indecomposable $\sEnd_R(M)$-free module. 
Moreover $\Im \, f (M)$ and $\Im \, g(M)$ are cyclic modules. 
This concludes that $H(M)$ must be nonzero. 
Therefore we have $H \in \Sigma (\XX \cup \YY)$ such that $H(M) \not=0$. 
This gives the contradiction that $M \in  \cl  (\XX \cup \YY )$, so that $M$ is in $\cl (\XX ) \cup \cl (\YY)$.  
\end{proof}

\begin{corollary}\label{A6}
The assignment $\XX\mapsto\cl(\XX)$ defines a topology on $\SpC$. 
That is a subset $\XX$ of $\SpC$ is closed if and only if $\cl(\XX)=\XX$. 

Particularly $\SpCM$ admits a topological structure concerning the topology.   
\end{corollary}

The author thanks Tsutomu Nakamura for telling him the remark below. 

\begin{remark}\label{A7}
Let $\mathrm{GProj}(R)$ be a category of Gorenstein-projective $R$-modules and let $\mathrm{GProj}(R)^{c}$ be the full subcategory consisting of compactly generated modules. 
It has been studied in \cite{Nak22} that the Ziegler spectrum is defined by using the functor category of the stable category of $\mathrm{GProj}(R)^{c}$. 
Suppose that $R$ is Gorenstein. 
Then it is shown in \cite[Theorem 2.33]{Nak22} that we have the triangulated equivalence $\sCMR \cong \underline{\mathrm{GProj}}(R)^{c}$. 
So if $R$ is Gorenstein, the spectrum $\SpCM$ is nothing but the Ziegler spectrum which is considered in \cite{Nak22} restricted to finitely generated ones. 
\end{remark}

We explore a topology on $\SpC$.

\begin{lemma}\label{A12}
Let $X, Y \in \SpC$ with $X \not\cong Y$. 
Suppose that $\sHomR (X, Y) \not=0$. Then $Y \not\in \cl (X)$. 
\end{lemma}

\begin{proof}
Take a generator $f_1 , \cdots ,f_n$ of $\sHomR (X, Y)$ as an $R$-module. 
We consider the functor induced by $X ^{\oplus n} \xrightarrow{(f_1 , \cdots .,f_n)} Y$, that is 
$$
\sHomR (-,X ^{\oplus n}) \xrightarrow{\sHomR (-, (f_1 , \cdots ,f_n))}  \sHomR (-, Y) \to F \to 0.
$$
Then $F(X) =0$ and $F(Y) \not=0$. 
This yields that $Y  \not\in \cl (X)$. 
\end{proof}

\begin{proposition}\label{A13}
We have $\cl (X) = \{ X\}$ for all $X \in \SpC$. 
Hence $\SpC$ is $T_1$-space. 
\end{proposition}

\begin{proof}
Let $Y \in \SpC$ which is not isomorphic to $X$. 
Suppose that $\sHom _R (X, Y)$ is not empty. 
Then $Y \not\in \cl (X)$  by Lemma \ref{A12}. 
Suppose that $\sHom _R (X, Y) = 0$. 
Then $\sHom _R (- ,Y)$ is contained in $\Sigma (X)$   
Assume that $Y \in \cl (X)$, and then $\sHomR (Y, Y)=0$. 
So that $Y$ is $0$ or $R$. 
This gives contradiction. 
It also indicates that $Y$ is not in $\cl (X)$.  
\end{proof}

From now we focus on $\CMR$, that is we consider the Ziegler spectrum on $\SpCM$.  
The connection between Auslander-Reiten (abbr. AR) sequences and isolated points in $\SpCM$ is well-known. 
The following result is a module version of \cite[Corollary 5.3.32]{Pr09}.

\begin{proposition}\label{A14}
Let $M \in \SpCM$. 
$M$ is an isolated point, that is $\{ M \}$ is open, if and only if there exists an AR sequence ending in $M$. 
\end{proposition}

\begin{proof}
If there exists an AR sequence ending in $M$ we can consider the functor $S_M$ which is obtained from the AR sequence. 
Then $\gamma (S_M) = \SpCM \backslash \{ M \}$ is closed, so that $\{ M \}$ is open.  

Suppose that $M$ is isolated, and then $\SpCM \backslash \{ M\}$ is closed. 
Since $\Sigma (\SpCM \backslash \{ M\})$ is not empty, we take $F \in \Sigma (\SpCM \backslash \{ M\})$. 
Then $F(M) \not=0$ and $F(N) = 0$ if $N \not \cong M$. 
By Yoneda's lemma, we have a nonzero morphism $\rho : \sHomR (-, M) \to F$. 
Since $\Im \rho$ is finitely presented and a subfunctor of $F$, by considering $\Im \rho$ instead of $F$, we may assume that $F$ has a presentation: 
$\sHomR (-, M) \to F \to 0$. 
Take a generator $f_1 , \cdots ,f_m$ of $\srad (M, M)$ which is the radical of $\sHom _R (M, M)$ as a right $\End _R (M)$-module. 
Then the image of $\sHomR (M, (f_1 , \cdots ,f_m)) : \sHomR (M, M^{\oplus m}) \to \sHomR (M, M)$ is $\srad (M, M)$. 
Notice that $\sHomR (M, M)/\srad (M, M) \cong k$ since $M$ is indecomposable.  
We consider the diagram below:
$$
\begin{CD}
0 @.0 @.\\
  @AAA @AAA\\
H_M @>>> S @>>>0 \\
@AAA @AAA \\
 \sHomR (-, M) @>{\rho}>> F@>>>0\\
@AA{f:=\sHomR (-, (f_1 , \cdots ,f_m))}A @AAA \\
\sHomR (-, M^{\oplus m}) @>>> \Im \rho \circ f  @>>>0\\
@. @AAA \\
@.0. @.\\
\end{CD}
$$
We should remark that $S = F/\Im \rho \circ f $ is finitely presented since $\Im \rho \circ f $ is so. 
Now we shall show $S$ is a simple functor with $S(M) \not= 0$. 
First we remark that $S(N) =0$ if $N \not\cong M$. 
Because we have $F(N) \to S(N) \to 0$, and $F(N)=0$ if if $N \not\cong M$. 
If we substitute $M$ into the diagram, then the diagram is one as a right $\End_R (M)$-module.  
Moreover $\rho (f (M) ) =  \rho (\srad (M, M) ) \subseteq \rad F(M)$. 
So we have $S(M) = F(M)/\rad F(M)$. 
If $S(M)=0$, then $F(M) = \rad F(M)$. 
By Nakayama's lemma, $F(M)=0$, which is a contradiction. 
If $S(M) \not=0$, $S(M)$ is isomorphic to $H_M(M) \cong k$. 
This yields that $S$ is a simple functor and we conclude that $M$ admits an AR sequence.
\end{proof}

We say that $(R, \m)$ is an isolated singularity if each localization $R_{\p }$ is regular for each prime ideal $\p$ with $\p \not= \m$. 
Remark that $R$ is an isolated singularity if and only if $\CMR$ admits AR sequences (cf. \cite[Theorem 3.2]{Y}). 
Thus Proposition \ref{A14} induces the following.

\begin{corollary}\label{A15}
The ring $R$ is an isolated singularity if and only if the topology of $\SpCM$ is discrete.
\end{corollary}

For a locally coherent category ${\mathfrak{C}}$, a bijective correspondence between closed subsets in $\mathsf{Sp}\, \mathfrak{C}$ and Serre subcategories in $\mod \,\mathfrak{C}$ is given. 
See \cite[Theorem 3.8.]{Her97} and \cite[Theorem 4.2.]{Kr97}.  
Unfortunately, in our setting, for a Serre subcategory $\FF \subseteq \modsC$, $\FF \not= \Sigma \circ \gamma (\FF )$ in general.

\begin{example}\label{A8}
Let $R=k[[x, y]]/(x^2)$. 
The indecomposable MCM $R$-modules are $R$, $I=(x)R$ and $I_n  = (x, y^n)R$ for $n \geq 1$. 
As calculated in \cite[Lemma 3.3.]{Hir21}, $\sHomR (X, I_n) \not=0$ for all $X \in \CMR$, that is $\gamma (\sHomR (-, I_n))= \emptyset$. 
Thus one has $\Sigma \circ \gamma (\sHomR (-, I_n))= \Sigma (\emptyset ) = \modsC$. 
However $\SS (\sHomR (-, I_n)) \not= \modsC$. 
Here we denote by $\SS (\sHomR (-, I_n))$ the smallest Serre subcategory which contains $\sHomR (-, I_n)$.  
Since $\KGdim \sHomR (-, I_n) = 1$ \cite[Proposition 3.8]{Hir21}, $\KGdim \SS (\sHomR (-, I_n)) =1$. 
Note that $\KGdim \sHomR (-, I) = 2$. \cite[Proposition 3.11]{Hir21}. 
Hence $\sHomR (-, I) \not\in \SS (\sHomR (-, I_n))$. 
Therefore we have $\SS (\sHomR (-, I_n)) \not= \modsC$.
\end{example}

We seek more about the correspondence. 
As mentioned the above example, $\FF \not= \Sigma \circ \gamma (\FF )$ for a Serre subcategory $\FF$ in $\modsCC$. 
Hence it seems to need an additional assumption to give the bijection. 
For a subcategory $\FF$ of an abelian category $\AA$, we define 
$$
{}^{\perp}\FF = \left\{ G \in \AA \ | \ \Hom _{\AA} (G, \FF ) = 0 \right\}, \quad
\FF ^{\perp} = \left\{ G \in \AA \ | \ \Hom _{\AA} (\FF, G ) = 0 \right\} . 
$$

\begin{proposition}\label{A9}
For a subset $\XX$ in $\SpC$, $({}^{\perp} \Sigma (\XX) )^{\perp} = \Sigma (\XX)$.
\end{proposition}

\begin{proof}
The inclusion $({}^{\perp} \Sigma (\XX) )^{\perp} \supseteq \Sigma (\XX)$ follows by the definition. 
For each $X \in \XX$, a functor $\sHomR(-,X)$ belongs to ${}^{\perp} \Sigma (\XX)$. 
Because $\Hom_{\modsCC} (\sHomR(-,X), F) \cong F(X) = 0$ by Yoneda's lemma for each $F \in \Sigma (\XX)$ and $X \in \XX$.   
Thus it follows from Yoneda's lemma again that $F (\XX) = 0$ for all $F \in ({}^{\perp} \Sigma (\XX) )^{\perp}$. 
Hence the inclusion  $({}^{\perp} \Sigma (\XX) )^{\perp} \subseteq \Sigma (\XX)$ holds. 
\end{proof}

Therefore we should add the assumption that $({}^{\perp} \FF )^{\perp} = \FF$ to give the bijective correspondence. 
We can obtain the following correspondence.

\begin{proposition}\label{A10}
If $R$ is an isolated singularity, we have a bijective correspondence between closed subsets $\XX$ in $\SpCM$ and  Serre subcategories $\FF$ of $\modsC$ with $({}^{\perp} \FF )^{\perp} = \FF$. 
\end{proposition}

To show the proposition, we need a lemma.

\begin{lemma}\label{A11}
Let $\FF$ be a Serre subcategory of $\modsC$ with $({}^{\perp} \FF )^{\perp} = \FF$. 
Assume that $R$ is an isolated singularity. 
For each functor $G \in {}^{\perp} \FF$, there exist $X \in add (\gamma (\FF ))$ such that $\sHomR ( -, X) \to G \to 0$. 
\end{lemma}

 \begin{proof}
 For simplicity we denote ${}^{\perp} \FF$ by $\GG$. 
 We have a presentation $\sHomR ( -, W) \to G \to 0$ for every $G \in \GG$. 
 Then we may assume that the presentation is minimal. 
 That is, for any direct summand $W'$ of $W$, $G(W') \not =0$. 
 Assume that there is a direct summand $X$ of $W$ such that $X$ is indecomposable and $X$ does not belong to $\gamma (F)$. 
As mentioned in \cite[Lemma 4.12]{Y}, we have an epimorphism $G \to S_X \to 0$ since $G(X) \not =0$. 
By the assumption of $X$ there exists $F \in \FF$ such that $F(X) \not=0$. 
Then we also obtain an epimorphism $F \to S_X \to 0$. 
Since $\FF$ is a Serre subcategory $S_X$ belongs to $\FF$. 
This implies that $\Hom (G', S_X) =0$ for all $G' \in \GG$ since $\GG ^{\perp} = ({}^{\perp} \FF )^{\perp} = \FF$. 
This is a contradiction, so that we obtain the assertion.      
 \end{proof}

 \begin{proof}[Proof of Proposition \ref{A10}]
 We show $\Sigma \circ \gamma (\FF ) = \FF$ for a Serre subcategory $\FF$ with $({}^{\perp} \FF )^{\perp} = \FF$. 
The inclusion $\supseteq$ is straightforward. 
Set $F$ a functor which belongs to $\Sigma \circ \gamma (\FF )$. 
Note that $F(\gamma (\FF ) ) = 0$. 
By Lemma \ref{A11}, for each $G \in {}^{\perp} \FF $, we have an epimorphism $\sHomR(-, X) \to G \to 0$ with $X \in add (\gamma (\FF ))$. 
Then we have the sequence:
$$
0 \to \Hom _{\modsC} (G, F) \to \Hom_{\modsC} (\sHomR ( -. X) , F). 
$$ 
By Yoneda's lemma $\Hom (\sHomR ( -, X) , F) \cong F(X)=0$. 
Hence $\Hom_{\modsC} (G, F) =0$, so that $F$ belongs to $({}^{\perp} \FF )^{\perp} = \FF$.  
\end{proof}

At the end of this section, we shall give some examples of closed subsets of $\SpC$. 
Let $\AA$ be an additive category and $\CC$ a subcategory. 
We say that $\CC$ is contravariantly finite in $\AA$ if for each $M \in \AA$ there exists $X \in \CC$ such that the restriction $\Hom_\AA (-, X) \to \Hom_\AA (-, M)$ to $\CC$ is surjective. 

\begin{example}\label{A16}
The following subsets are closed subsets in $\SpC$. 
\begin{itemize}
\item[(1)] Every finite subsets are closed. 
\item[(2)] For each subcategory $\FF$ in $\modsCC$, $\gamma (\FF )$ is closed. 
\item[(3)] If $R$ is an isolated singularity then every subsets are closed. 
\item[(4)] If $add (\XX )$ is contravariantly finite in $\sCC$, then $\XX$ is closed. 
\end{itemize}
\end{example}

\begin{proof}
\begin{itemize}
\item[(1)] By Proposition \ref{A13} , $\SpC$ is a $T_1$-space. 
Hence every finite subsets are closed. 
\item[(2)] According to Lemma \ref{A5}(4), $\gamma \circ \Sigma ( \gamma (\FF ) )= \gamma (\FF )$. 
\item[(3)] Suppose that $R$ is an isolated singularity. 
For a subset $\XX$ in $\SpCM$. 
Consider the subcategory $\FF _{\XX}$ consisting simple functors $S_Y$ for all $Y \in \SpCM \backslash \XX$. 
Then $\XX = \gamma (\FF _{\XX})$. 
By (2), $\XX$ is closed. 
\item[(4)] For an $M \in \SpC \backslash \XX$. 
Since $add (\XX )$ is contravariantly finite, we have an exact sequence of functors:
$$
\sHom_R (-, X) \xrightarrow{\sHom_R(-, f)} \sHom_R (-, M) \to F_M \to 0,
$$
where $X$ in $add (\XX )$, $f \in \sHom_R(X,M)$ and $F(\XX ) = 0$. 
Note here that $F(M) \not=0$. 
Assume that $F(M) =0$. 
Then there exists a morphism $g \in \sHom_R(M, X)$ such that $1_M = f \circ g$. 
Thus $f$ is split, so that $M$ is a direct summand of $X$. 
This yields that $M$ belongs to $add (\XX )$. 
It makes a contradiction. 
Consider a subcategory $\FF = \{ F_M \ | \ M \in \SpC \backslash \XX \}$. 
Then one has $\XX = \gamma (\FF )$. 
By (2), $\XX$ is a closed subset. 
\end{itemize}
\end{proof}

\section{Cantor-Bendixson rank}\label{B}

In this section, we calculate a Cantor-Bendixson rank of $\SpCM$. 
The Cantor-Bendixson rank measures how far the topology is from the discrete topology. 
 
\begin{definition}[Cantor-Bendixson rank]\label{B1}\cite[5.3.6]{Pr09}
Let $\TT$ be a topological space. 
If $x \in \TT$ is an isolated point, then $\CB (x) = 0$. 
Put $\TT ' \subset \TT$ is a set of the \underline{\bf non}-isolated point. 
Define the induced topology on $\TT '$. 
Set $\TT ^{(0)} = \TT, \TT ^{(1)} = \TT ^{(0)'}, \cdots, \TT ^{(n+1)} = \TT ^{(n)'}$. 
We define $\CB (x) = n$ if $x \in \TT^{(n)} \backslash \TT^{(n+1)}$
If there exists $n$ such that $\TT^{(n+1)} = \emptyset$ and $\TT^{(n)} \not= \emptyset$, then $\CB ( \TT ) = n$. 
Otherwise $\CB (\TT )= \infty$.
\end{definition}

\begin{example}\label{B2}
Let $R$ be a DVR (e,g. $R=k[[ x ]]$). 
Then $\CB (\Spec R ) = 1$ concerning the Zariski topology. 
Note that $\Spec R  = \{ (0), \m \}$. 
$(0)$ is an isolated point since $D(f) = \{ (0)\}$ for some $f  \in R \backslash \{ 0 \}$. 
Thus $\Spec R ' = \{ \m \} = \Spec R ^{(1)} $, and $\m$ is isolated in the induced topology. 
In the case $R=k[[ x, y]]$, you can show that $\CB (\Spec R)= \infty$. 
Note that $\Spec R ' = \Spec R$. 
\end{example}

By Corollary \ref{A15}, we know $\SpCM$ is a discrete topology if and only if $R$ is an isolated singularity. 

\begin{corollary}\label{B3}
If $R$ is an isolated singularity if and only if $\CB (\SpCM ) = 0$. 
\end{corollary}

We say that $R$ is of finite (countable) $\mathrm{CM}$-representation type if there exists only finitely (countably) many isomorphism classes of indecomposable MCM modules. 
If $R$ is of finite $\mathrm{CM}$-representation type, $R$ is an isolated singularity (e.g. \cite[Theorem 4.22]{Y}). 
Hence as another consequence of Corollary \ref{B3}, we obtain the following. 

\begin{corollary}\label{B4}
If $R$ is of finite $\mathrm{CM}$-representation type then $\CB (\SpCM ) = 0$. 
\end{corollary}

\begin{remark}\label{B5}
The converse of Corollary \ref{B4} does not hold. 
For instance let $R=k[x,y]/(x^2, y^2)$. 
Then $R$ is an isolated singularity but it is not of finite representation type. 
 \end{remark}

For the Cantor-Bendixson rank, we shall consider a slightly wider class of rings.

\begin{definition}\cite{KLT20}\label{B6}
We say that a Cohen--Macaulay local ring $R$ is $\mathrm{CM}_+$-finite if there exist only finitely many isomorphism classes of indecomposable MCM modules that are \underline{\bf not} locally free on the punctured spectrum. 
\end{definition}

\begin{example}\label{B7}
The following rings are $\mathrm{CM}_+$-finite. 
\begin{itemize}
\item[(1)] A ring which is an isolated singularity. 
Thus a ring which is of finite $\mathrm{CM}$-representation type is $\mathrm{CM}_+$-finite (cf. \cite[Lemma 3.3, Theorem 4.22]{Y}). 
 
\item[(2)] A hypersurface ring which is of countable $\mathrm{CM}$-representation type (cf. \cite{AIT12}.)
\end{itemize}
\end{example}

\begin{theorem}\label{B8}
If $R$ is $\mathrm{CM}_+$-finite then $\CB (\SpCM  ) \leq 1$. 
\end{theorem}

\begin{proof}
We denote by $\CMR _0$ the subset of $\SpCM$ consisting of modules that are locally free on the punctured spectrum and put $\CMR _{+} := \SpCM \backslash \CMR_0$. 
For all $M \in \CMR_0 $, $M$ is an isolated point since $M$ admits an AR sequence. 
Thus $\CB (\CMR _0) = 0$. 

On the other hand, for all $M \in \CMR_+ $, $M$ is not isolated. 
Since $R$ is $\mathrm{CM}_+$-finite, $\CMR_+ $ is a finite set. 
Hence, for each $M \in \CMR_+ $,
$$
V_M := \bigcup_{X \not=M,  X \in \CMR_+ } \cl (X)
$$
is a finite union, so that it is closed in $\SpCM$.  
Thus  $ \CMR_+  \bigcap \left[\SpC \backslash V_M \right] = \{ M \}$ is open in the induced topology $\CMR_+ \cap \SpCM$. 
Hence we have $M \in \SpCM^{(1)} \backslash \SpCM^{(2)}$ for all $M \in \CMR_+ $. 
Therefore $\CB (\SpCM ) \leq 1$. 
\end{proof} 
\subsection*{Acknowledgment}
The author was partly supported by JSPS KAKENHI Grant Number 21K03213.
He would like to express his deep gratitude to Ryo Takahashi and Yuji Yoshino for their valuable discussions and helpful comments. 
He also thanks Tsutomu Nakamura for telling him Remark \ref{A7}.


\begin{thebibliography}{1}

\bibitem{AIT12}
{Araya, T}, {Iima, K.-i.}, {Takahashi, T.}: 
{\it On the structure of Cohen--Macaulay modules over hypersurfaces of countable Cohen--Macaulay representation type}. 
J. Algebra. {\bf 361}, 213--224 (2012).
  
\bibitem{Her97}
{Herzog, I}: 
{\it The Ziegler spectrum of a locally coherent Grothendieck category}.
Proc. London Math. Soc. (3) {\bf 74} (1997), no. 3, 503--558.

\bibitem{Hir21}
{Hiramatsu, N.}: 
{\it Krull--Gabriel dimension of Cohen--Macaulay modules over hypersurfaces of countable Cohen--Macaulay representation type}.
arXiv:2112.13504.

\bibitem{KLT20}
{Kobayashi, T.}, {Lyle, J.} and {Takahashi, R.}: 
{\it Maximal Cohen--Macaulay modules that are not locally free on the punctured spectrum}, 
J. Pure Appl. Algebra {224}, (2020).

\bibitem{Kr97}
{Krause, H}: 
{\it The spectrum of a locally coherent category}. 
J. Pure Appl. Algebra {\bf 114} (1997), no. 3, 259--271.


\bibitem{LP19}
{Los, I} and {Puninski, G.}: 
{\it The Ziegler spectrum of the D-infinity plane singularity}. 
Colloq. Math. {\bf 157} (2019), no. 1, 35--63.

\bibitem{Nak22}
{Nakamura, T.}: 
{\it Indecomposable pure-injective objects in stable categories of Gorenstein-projective modules over Gorenstein orders}.
arXiv:2209.15630.


\bibitem{Pr09}
{Prest, M.}:
{\it Purity, spectra and localisation}, 
Encyclopedia of Mathematics and its Applications {\bf 121}, 
Cambridge University Press, Cambridge, 2009, xxviii+769. 

\bibitem{P18}
{Puninski, G.}:
{\it The Ziegler Spectrum and Ringel's Quilt of the A-infinity Plane Curve Singularity}, 
Algebr Represent Theor {\bf 21}, 419--446 (2018).


\bibitem{Y}
{Yoshino, Y.}:
{\it Cohen--Macaulay Modules over Cohen--Macaulay Rings, London Mathematical Society}, Lecture Note Series, {\bf 146}, { Cambridge University Press, Cambridge}, 1990. 


\bibitem{Y05}
{Yoshino, Y.}:
{\it A functorial approach to modules of G-dimension zero}, 
Illinois J. Math. {\bf 49} (2) 345--367, Summer 2005. 

\bibitem{Z84}
{Ziegler, M. }:
{\it Model theory of modules}, 
Ann. Pure Appl. Logic {\bf 26} (1984), no.2, 149--213.
\end{thebibliography}
\end{document}